\documentclass[12pt]{article}

\usepackage{tikz}
\usepackage{subfigure}
\usepackage[english]{babel}

\usepackage[center]{caption2}
\usepackage{amsfonts,amssymb,amsmath,latexsym,amsthm}
\usepackage{multirow}

\def\ex{\mathrm{ex}}

\def\lf{\left\lfloor}
\def\rf{\right\rfloor}
\def\lc{\left\lceil}
\def\rc{\right\rceil}

\def\Fkr{F_{k,r}}
\def\geqs{\geqslant}
\def\leqs{\leqslant}

\def\Dta{\Delta}

\newtheorem{thm}{Theorem}[section]

\newtheorem{lem}[thm]{Lemma}

\newtheorem{conj}[thm]{Conjecture}

\newtheorem{claim}{Claim}

\newtheorem{obs}[thm]{Observation}

\begin{document}
\title{Tur\'{a}n number and decomposition number of intersecting odd cycles\thanks{The work was supported by NNSF of China (No. 11671376) and the Fundamental Research Funds for the Central Universities.}}
\author{Xinmin Hou$^a$, \quad Yu Qiu$^b$, \quad Boyuan Liu$^c$\\
\small $^{a,b,c}$ Key Laboratory of Wu Wen-Tsun Mathematics\\
%\small Chinese Academy of Sciences\\
\small School of Mathematical Sciences\\
\small University of Science and Technology of China\\
\small Hefei, Anhui 230026, China.\\
%\small $^a$xmhou@ustc.edu.cn,\quad  $^b$yuqiu@mail.ustc.edu.cn\\
%\small $^c$lby1055@mail.ustc.edu.cn
}

\date{}

\maketitle

\begin{abstract}
An extremal graph for a given graph $H$ is a graph on $n$ vertices with maximum number of edges that does not contain $H$ as a subgraph. Let $s,t$ be integers and let $H_{s,t}$ be a graph consisting of $s$ triangles and $t$ cycles of odd lengths at least 5 which intersect in exactly one common vertex.
%Let $\ex(n, H)$ be the
Erd\H{o}s et al. (1995) determined the extremal graphs for $H_{s,0}$. Recently, Hou et al. (2016) determined the extremal graphs for $H_{0,t}$, where the $t$ cycles have the same odd length $q$ with $q\ge 5$. In this paper, we further determine the extremal graphs for $H_{s,t}$ with $s\ge 0$ and $t\ge 1$. Let $\phi(n,H)$ be the largest integer such that, for all graphs $G$ on $n$ vertices, the edge set $E(G)$ can be partitioned into at most $\phi(n, H)$ parts, of which every part either is a single edge or forms a graph isomorphic to $H$. Pikhurko and Sousa conjectured that $\phi(n,H)=\ex(n,H)$ for $\chi(H)\geqs3$ and all sufficiently large $n$.  Liu and Sousa (2015) verified the conjecture for $H_{s,0}$.  In this paper,
%we determine the extremal graphs for intersecting odd cycles that may have different lengths and
we further verify Pikhurko and Sousa's conjecture for $H_{s,t}$ with $s\ge 0$ and $t\ge 1$.
\end{abstract}

\section{Introduction}

In this paper, all graphs considered are simple and finite. For a graph $G$ and a vertex $x\in V(G)$, the neighborhood of $x$ in $G$ is denoted by $N_G(x)$. The {\it degree} of $x$, denoted by $\deg_G(x)$, is $|N_G(x)|$. Let  $\delta(G),\Delta(G)$ and $\chi(G)$ denote the minimum  degree, maximum degree and chromatic number of $G$, respectively. Let $e(G)$ be the number of edges of $G$. For a graph $G$ and $S, T\subset V(G)$, let $e_G(S, T)$ be the number of edges $e=xy\in{E(G)}$ with $x\in S$ and $y\in T$, if $S=T$, we use $e_G(S)$ instead of $e_G(S, S)$, and  $e_G(u, T)$ instead of $e_G(\{u\}, T)$ for convenience, the index $G$ will be omitted if no confusion from the context. For a subset $X\subseteq V(G)$ or $X\subseteq  E(G)$, let $G[S]$ be the subgraph of $G$ induced by $X$, that is $G[X]=(X, E(X))$ if $X\subseteq V(G)$, or $G[X]=(V(X),X)$ if $X\subseteq  E(G)$. A {\it matching} $M$ in $G$ is a subset of $E(G)$ with $\delta(G[M])=\Delta([M])=1$. The {\it matching number} of $G$, denoted by $\nu(G)$, is the maximum number of edges in a matching in $G$. A maximum cut of $G$ is a bipartition of $V(G)=V_0\dot{\cup}V_1$ such that $e_G(V_0,V_1)$ is maximized. For $x\in V(G)$ and $A\subset V(G)$, let $E_{A}(x)=\{e\in E(G[A])\ |\  V(e)\cap N_G(x)\neq\emptyset \}$.
A cycle of length $q$ is called a {\it $q$-cycle}. {Given a partition of $V(G)=V_0\dot\cup V_1$ and $x\in V_i$ ($i=0,1$),  $\deg_{G[V_i]}(x)$ is called the {\it in-degree } of $x$, similarly, we call $e_G(x,V_{1-i})$ the {\it out-degree} of $x$.}

Given two graphs $G$ and $H$, we say that $G$ is {\it $H$-free} if $G$ does not contain an $H$ as a subgraph. The Tur\'an number, denoted by $\ex(n,H)$, is the largest number of edges of an $H$-free graph on $n$ vertices. That is,
$$\ex(n,H)=\max\{e(G):|V(G)|=n,  \mbox{ $G$ is $H$-free}\}.$$
For positive integers $n$ and $r$ with $n\geqs r$, the Tur\'an graph, denoted by $T_{n,r}$, is the balanced complete $r$-partite graph on $n$ vertices, where each part has size $\lfloor\frac nr\rfloor$ or $\lceil\frac nr\rceil$.

Let $s, t$ be integers and let $H_{s,t}$ be a graph consisting of $s$ triangles and $t$ cycles of odd lengths at least 5 which intersect in exactly one common vertex, called the center of $H_{s,t}$.
%that consist of $s+t$ odd cycles intersecting at exactly one common vertex, called {\em{center}}, where these $s+t$ odd cycles consist of $s$ triangles and $t$ other odd cycles each of length at least $5$. It's clear that $C_{k,q}\in\mathcal C_{0,k}$ for odd $q\geqs5$.
%A {\it $k$-fan}, denoted by $F_k$, is a graph on $2k+1$ vertices consisting of $k$ triangles which intersect in exactly one common vertex.
In 1995, Erd\H{o}s et al.~\cite{Erdos-95} determined the value of $ex(n, H_{k, 0})$ and the extremal graphs for $H_{k,0}$.
\begin{thm}[\cite{Erdos-95}]\label{THM:Fk}
	For $k\geqs1$ and $n\geqs50k^2$,$$\ex(n,H_{k,0})=e(T_{n,2})+g(k),$$
	where
	$$g(k)=\left\{
	\begin{aligned}
	&k^2-k & &\mbox{if $k$ is odd,}\\
	&k^2-\frac{3}{2}k & &\mbox{if $k$ is even.}\\
	\end{aligned} \right.$$
	Moreover, when $k$ is odd, the extremal graph must be a $T_{n,2}$ with two vertex disjoint copies of $K_k$ embedding in one partite set. When $k$ is even, the extremal graph must be a $T_{n,2}$ with a graph having $2k-1$ vertices, $k^2-\frac{3}{2}k$ edges with maximum degree $k-1$ embedded in one partite set.
\end{thm}
In 2003, Chen et al.~\cite{Wei} generalized Erd\H{o}s et al.'s result to $ex(n, \Fkr)=e(T_{n,r-1})+g(k)$, where $\Fkr$ is a graph consisting of $k$ complete graphs of order $r(\ge 3)$ which intersect in exactly one common vertex and $g(k)$ is the same as in Theorem~\ref{THM:Fk}. The above result were further generalized by Gelbov (2011) and Liu (2013). They determined the extremal graphs for blow-ups of paths~\cite{Glebov-2011}, cycles and a large class of trees~\cite{Liu-2013}. Recently, Hou et al.~(2016) generalized Erd\H{o}s et al.'s result in another way, they determined the extremal graphs for a special family of $H_{0,k}$, where the $k$ odd cycles have the same length $q$ with $q\ge 5$ (denoted by $C_{k,q}$ in~\cite{Ex for Ckq}).

%It's also nutural to generalize $k$-fan to interecting odd cycles. Precisely, for integers $k\geqs2$ and odd $q\geqs5$, let $C_{k,q}$ be the graph consisting of $k$ odd cycles of length $q$ that intersect exactly at one common vertex. Earlier in \cite{Ex for Ckq}, we determined $\ex(n,C_{k,q})$ and corresponding extremal graph as follows.

\begin{thm}[\cite{Ex for Ckq}]\label{THM:Ex for Ckq}
	For an integer $k\geqs2$ and an odd integer $q\geqs5$, there exists $n_0(k,q)\in\mathbb{N}$ such that for all $n\geqs n_0(k,q)$, we have $$\ex(n,C_{k,q})=e(T_{n,2})+(k-1)^2,$$   and the only extremal graph is a $T_{n,2}$ with a $K_{k-1,k-1}$ embedded in one partite set.
\end{thm}

As we have seen from Theorems~\ref{THM:Fk} and \ref{THM:Ex for Ckq} that the extremal graphs for $H_{k,0}$ and $H_{0,k}$  are different. A natural and interesting problem is to determine the extremal graphs for  mixed graph $H_{s,t}$. In this paper, our first main result solves the problem.
 %More generally, in this paper, we consider all intersecting odd cycles that may have different lengths. To be precise, for non-negative integers $s$ and $t$,

Let $\mathcal C_{s,t}$ be the family of all graphs $H_{s,t}$ and let $\mathcal{F}_{n,s,t}$ be the family of graphs with each member is a Tur\'an graph $T_{n,2}$ with a graph $H$ embedded in one partite set, where
$$H=\left\{
\begin{aligned}
&K_{s+t-1,s+t-1} & &\mbox{if $(s,t)\neq(3,1)$,}\\
&\mbox{$K_{3,3}$ or $3K_3$} & &\mbox{if $(s,t)=(3,1)$,}
\end{aligned} \right.$$
where $3K_3$ is the union of three disjoint triangles.
%The subfamily of $\mathcal{F}_{n,3,1}$ when $H=3K_3$ is denoted by $\mathcal{F'}_{n,3,1}$.
%It's clear that $C_{k,q}\in\mathcal C_{0,k}$ for odd $q\geqs5$.

\begin{thm}\label{THM:MAIN THEOREM 1}
 For any integers $s\geqs0,t\geqs1$ and for any $H_{s,t}\in\mathcal C_{s,t}$, there exists $n_1(H_{s,t})\in\mathbb{N}$ such that for all $n\geqs n_1(H_{s,t})$, $$\ex(n,H_{s,t})=e(T_{n,2})+(s+t-1)^2,$$ and the only extremal graphs for $H_{s,t}$ are members of $\mathcal{F}_{n,s,t}$.
\end{thm}

A parameter related to Tur\'{a}n number is the so called decomposition number. Given two graphs $G$ and $H$, an {\em $H$-decomposition} of $G$ is a partition of edges of $G$ such that every part is a single edge or forms a graph isomorphic to $H$. Let $\phi(G,H)$ be the smallest number of parts in an $H$-decomposition of $G$. Clearly, if $H$ is non-empty, then
$$\phi(G,H)=e(G)-p_{H}(G)(e(H)-1),$$
where $p_{H}(G)$ is the maximum number of edge-disjoint copies of $H$ in $G$. Define
$$\phi(n,H)=\max\{\phi(G,H):|V(G)|=n \}.$$

This function, motivated by the problem of representing graphs by set intersections, was first studied by Erd\H{o}s, Goodman and P\'{o}sa \cite{intersections}, they proved that $\phi(n,K_3)=\ex(n,K_3)$. The result was generalized to $\phi(n,K_r)=\ex(n,K_r)$, for all $n\geqs{r}\geqs3$ by Bollob\'{a}s~\cite{Bollobas}.
More generally,  Pikhurko and Sousa~\cite{general} proposed the following conjecture.
\begin{conj}[\cite{general}]\label{CONJ: Pikhurko and Sousa}
	For any graph $H$ with $\chi(H)\geqs3$, there is an integer $n_0=n_0(H)$ such that $\phi(n,H)=\ex(n,H)$ for all $n\geqs{n_0}$.
\end{conj}
In~\cite{general}, Pikhurko and Sousa also proved that $\phi(n,H)=\ex(n,H)+o(n^2).$ The error term was improved to be $O(n^{2-\alpha})$ for some $\alpha>0$ by Allen, B\"{o}ttcher, and Person~\cite{ImprovedError}. Sousa verified the conjecture for some families of edge-critical graphs, namely, clique-extensions of order $r\geqs4$ $(n\geqs{r})$~\cite{clique-extension} and the cycles of length 5 $(n\geqs6)$~\cite{5-cycles} and 7 $(n\geqs10)$~\cite{7-cycle}.
In~\cite{edge-critical case}, \"{O}zkahya and Person verified the conjecture for all edge-critical graphs with chromatic number $r\geqs3$. Here, a graph $H$ is called {\em edge-critical}, if there is an edge $e\in{E(H)}$, such that $\chi(H)>\chi(H-e)$. For non-edge-critical graphs, Liu and Sousa~\cite{k-fan} verified the conjecture for $H_{k,0}$ and recently, the result was generalized to $F_{k,r}$ for all $k\geqs2$ and $r\geqs3$ by Hou et al.~\cite{Dec for Fkr}.

%\begin{thm}[\cite{k-fan}]\label{THM: Liu-Sousa} For $k\geqs1$, there exists $n_0=n_0(k)$ such that $\phi(n,F_{k})=\ex(n,F_{k})$ for all $n\geqs{n_0}$. Moreover, the only graphs attaining $\ex(n,F_{k})$ are the extremal graphs for $F_{k}$.
%\end{thm}

%In 2015, we generalize this result to $F_{k,r}$ for all $k\geqs2$ and $r\geqs3$ (see \cite{Dec for Fkr}).

%In this paper, motivated by the results in~\cite{Erdos-95,Glebov-2011,Liu-2013,Bollobas} and based on earlier work in \cite{Ex for Ckq}, we generalize Theorem \ref{THM:Fk}, \ref{THM:Ex for Ckq} and \ref{THM: Liu-Sousa} to  all members in $\mathcal C_{s,t}$ where $s\geqs0$ and $t\geqs1$. Let $3K_3$ be a graph consisting of three vertex-disjoint triangles.
Our second main result verifies that Pikhurko and Sousa's conjecture is true for $H_{s,t}$ with $s\ge 0$ and $t\ge 1$.

\begin{thm}\label{THM:MAIN THEOREM 2}
	For any integer $s\geqs0,t\geqs1$ and for any $H_{s,t}\in\mathcal C_{s,t}$, there exists $n_2(H_{s,t})\in\mathbb{N}$ such that for all $n\geqs n_2(H_{s,t})$, $$\phi(n,H_{s,t})=\ex(n,H_{s,t}).$$ Moreover the only graphs attaining $\ex(n,H_{s,t})$ are members of $\mathcal{F}_{n,s,t}$.
\end{thm}

The remaining of the paper is arranged as follows. Section 2 gives all the technical lemmas we need. Sections 3 and 4 give the proofs of Theorems~\ref{THM:MAIN THEOREM 1} and~\ref{THM:MAIN THEOREM 2}, respectively.

\section{Lemmas}
The following two lemmas due to Chav\'atal and Hanson~\cite{Hanson} and Erd\H{o}s et al.~\cite{Erdos-95} are used to evaluate the maximum number of edges of a graph with given maximum degree and matching number.
\begin{lem}[\cite{Hanson}]\label{LEMMA:f_Nu_Delta} For any graph $G$ with maximum degree $\Dta\geqs1$ and matching number $\nu\geqs1$, we have $e(G)\leqs f(\nu,\Delta)=\nu\Dta+{\lf\frac{\Dta}{2}\rf}{\lf\frac{\nu}{\lc\Dta/2\rc}\rf}\leqs \nu(\Delta+1)$.
\end{lem}

%The following Lemma can be found in \cite{Erdos}.
\begin{lem}[\cite{Erdos-95}]\label{LEMMA: Sum min deg b leqs}
	Let $H$ be a graph with maximum degree $\Delta$ and matching number $\nu$ and let $b$ be a nonnegative integer such that $b\leqs\Delta(H)-2$. Then$$\sum_{x\in V(H)}\min\{\deg_H(x),b\}\leqs\nu(\Delta+b).$$
\end{lem}

The following two stability lemmas due to Erd\H{o}s~\cite{Erdos-66}, Simonovits~\cite{Simonovits-66},  \"{O}zkahya and Person \cite{edge-critical case} play  an important role to determine the Tur\'an number and decomposition number of a given graph $H$.

\begin{lem}[\cite{Erdos-66,Simonovits-66}]\label{LEMMA:partition} Let $H$ be a graph with $\chi(H)=r\geqs3$ and $H\neq{K_r}$. Then, for every $\gamma>0$, there exists $\delta>0$ and $n_0=n_0(H,\gamma)\in\mathbb{N}$ such that the following holds. If $G$ is an $H$-free graph on $n\geqs{n_0}$ vertices with $e(G)\geqs\ex(n,H)-\delta{n^2},$ then there exists a partition of $V(G)=V_1\dot{\cup}\cdots\dot{\cup}V_{r-1}$ such that $\sum^{r-1}_{i=1}e(V_i)<\gamma n^2$.
\end{lem}

\begin{lem}[\cite{edge-critical case}]\label{LEMMA: stability for phi}Let $H$ be a graph with $\chi(H)=r\geqs3$ and $H\neq{K_r}$. Then, for every $\gamma>0$, there exists $\delta>0$ and $n_0=n_0(H,\gamma)\in\mathbb{N}$ such that the following holds. If $G$ is a graph on $n\geqs{n_0}$ vertices with $\phi(G,H)\geqs\ex(n,H)-\delta{n^2},$ then there exists a partition of $V(G)=V_1\dot{\cup}\cdots\dot{\cup}V_{r-1}$ such that $\sum^{r-1}_{i=1}e(V_i)<\gamma n^2$.	
\end{lem}

The following lemma can be found in \cite{Ex for Ckq}.
\begin{lem}[Lemma 8 in~\cite{Ex for Ckq}]\label{Lemma:delete to get large minidegree}
	Let $n_0$ be an integer and let $G$ be a graph on $n\geqs n_0+{n_0\choose 2}$ vertices with $e(G)=e(T_{n,2})+j$ for some integer $j\geqs0$. Then $G$ contains a subgraph $G'$ on $n'> n_0$ vertices  such that $\delta(G')\geqs\delta(T_{n',2})$ and $e(G')\geqs e(T_{n',2})+j+n-n'$.
\end{lem}

The proof of the following lemma is almost the same as  the proof of Lemma 6 in~\cite{Dec for Fkr}, we give the proof here for completeness.
\begin{lem}\label{Lemma:phi case delete to get large minidegree}
	Let $n_0$ be an integer and $H$ be a given graph with $\chi(H)=r\ge 3$ and $\ex(n,H)-\ex(n-1,H)\geqs\delta(T_{r-1,n})$ for all $n\geqs n_0$.
 Let $G$ be a graph on $n> n_0+{n_0\choose 2}$ vertices with $\phi(G,H)=\ex(n,H)+j$ for some integer $j\geqs0$. Then $G$ contains a subgraph $G'$ on $n'> n_0$ vertices such that $\delta(G')\geqs\delta(T_{n',r-1})$ and $\phi(G',H)\geqs e(T_{n',r-1})+j+n-n'$.
\end{lem}
\begin{proof}
If $\delta(G)\ge \delta(T_{n,r-1})$, then $G$ is the desired graph and we have nothing to do. So assume that $\delta(G)<\delta(T_{n,r-1})$.
Let $v\in V(G)$ with $\deg_G(v)<\delta(T_{n,r-1})$ and set $G_1= G-v$. Then $\phi(G_1,H)\geqs\phi(G,H)-\deg_G(v)\geqs \ex(n,H)+j-\delta(T_{n,r-1})+1\geqs\ex(n-1, H)+j+1$, since $\ex(n,H)-\ex(n-1,H)\geqs\delta(T_{n,r-1})$. We may continue this procedure until we get a graph $G'$ on $n-i$ vertices with $\delta(G')\geqs\lf\frac{r-2}{r-1}(n-i)\rf$ for some $i<n-n_0$, or until $i=n-n_0$. But the latter case can not occur since $G'$ is a graph on $n_0$ vertices with $e(G')\geqs\phi(G',H)\geqs\ex(n_0,H)+j+i\geqs n-n_0>\binom{n_0}{2}$, which is impossible.
\end{proof}

The following observation was given in~\cite{Ex for Ckq}.
\begin{obs}[Observation 5 in~\cite{Ex for Ckq}]\label{OBS: o1}
Let $G$ be a graph with no isolated vertex. If $\Delta(G)\leqs 2$, then $$\nu(G)\geqs \frac{|V(G)|-\omega(G)}2,$$ where $\omega(G)$ is the number of components of $G$.
\end{obs}
%\begin{proof}
%Since $\Delta(G)\leqs 2$, each component of $G$ is a path or a cycle. Hence each component $C$ of $G$ has matching number at least $\frac{|V(C)|-1}2$. This implies the desired result.
%\end{proof}

The following is a technical lemma to determine the extremal graphs for intersecting odd cycles.
\begin{lem}\label{LEMMA:MAIN LEMMA}
Let $s\geqs0$ and $t\geqs1$ be two integers and $k=t+s$. Let $G$ be a graph with no isolated vertex and $\nu(G)\leqs k-1$. If for all $x\in V(G)$ with $\deg(x)\geqs s$,  we have $\deg(x)+\nu(G-N(x))\leqs k-1$, then $e(G)\leqs (k-1)^2$. Moreover, equality holds if and only if $G=K_{k-1,k-1}$ or $G=3K_3$, the latter case happens only if $s=3$ and $t=1$.
\end{lem}

\begin{proof}
	Note that the conditions of the lemma imply that $\Delta(G)\le k-1$ and $k\geqs2$. %We split the proof into following two cases.

\medskip	
\noindent{\bf Case 1.} $\Delta(G)\leqs k-2$.

\noindent
Then $k\geqs3$ in this case. By Lemma \ref{LEMMA:f_Nu_Delta},  we have $e(G)\leqs{f(k-1, k-2)}=(k-1)(k-2)+{\lf\frac{k-2}{2}\rf}{\lf\frac{k-1}{\lc (k-2)/2\rc}\rf}\le (k-1)^2$, and the equality holds only if $\nu=k-1$, $\Delta=k-2$ and $k=4$. Now assume that $e(G)=f(3, 2)=3^2=9$. Since $\Delta(G)\leqs 2$ and $G$ has no isolated vertex, $|V(G)|\geqs e(G)=9$ (the equality holds if and only if $G$ is 2-regular) and $\omega(G)\leqs \nu(G)=3$. By Observation~\ref{OBS: o1}, $$3=\nu(G)\geqs \frac{|V(G)|-\omega(G)}2\geqs \frac{|V(G)|-3}2.$$ Hence $|V(G)|\leqs 9$. Thus $|V(G)|=9$ (and so $G$ is 2-regular) and $\omega(G)=3$.  Therefore, $G=3K_3$. Then, for any $x\in{V(G)}$, $\deg(x)+\nu(G-N(x))=2+2>3=k-1$. So, by the condition of the lemma,  $2=\deg(x)<s=4-t\le 3$ (since $t\geqs1$). Therefore,   we must have $s=3$ and $t=1$.

\medskip	
\noindent{\bf Case 2.} $\Delta(G)=k-1$.

 \noindent
Choose $x\in{V(G)}$ such that $\deg(x)=k-1$. Then $\nu(G-N(x))=0$. Hence $e(G-N(x))=0$, that is $V(G)\setminus N(x)$ is an independent set of $G$. Let $N(x)=\{x_1,\cdots,x_{k-1}\}$. For each $i\in [1,k-1]$, denote $d_i=\deg(x_i)$ and $\tilde{d_i}=\deg_{G[N(x)]}(x_i)$. Then
	\begin{eqnarray*}
		e(G)&=&e(G[N(x)])+e(N(x),V(G)\setminus N(x))=\frac12\sum_{i=1}^{k-1}\tilde{d_i}+\sum_{i=1}^{k-1}(d_i-\tilde{d_i})\\
		&=&\sum_{i=1}^{k-1}d_i-\frac12\sum_{i=1}^{k-1}\tilde{d_i}\leqs (k-1)^2-\frac12\sum_{i=1}^{k-1}\tilde{d_i}\leqs (k-1)^2,
	\end{eqnarray*}
and  the equality holds if and only if $d_i=k-1$ and $\tilde{d_i}=0$ for each $i\in [1,k-1]$, that is $G$ is a bipartite graph with partite sets $N(x)=\{x_1, \cdots, x_{k-1}\}$ and $V(G)\setminus N(x)$. To show that $G=K_{k-1,k-1}$, it suffices to prove that $|V(G)\setminus N(x)|=k-1$. If $|V(G)\setminus N(x)|>k-1$, then there must exist a vertex $y\in (V(G)\setminus N(x))\setminus N(x_1)$ since $\deg(x_1)=d_1=k-1$. Since $G$ has no isolated vertex, $y$ must be adjacent to some vertex $x_j$ with $j\not= 1$. This implies that  $\nu(G-N(x_1))\geqs 1$. Hence we have $\deg(x_1)+\nu(G-N(x_1))\ge k$, but $\deg(x_1)=k-1\geqs s$, a contradiction to  $\deg(x_1)+\nu(G-N(x_1))\leqs k-1$.
\end{proof}

The following lemma states that the members of $\mathcal{F}_{n,s,t}$ are actually $H_{s,t}$-free.

\begin{lem}\label{LEMMA:Gnk is Ckq-free}
Each member of $\mathcal{F}_{n,s,t}$ is $H_{s,t}$-free for any $H_{s,t}\in\mathcal C_{s,t}$.
\end{lem}
\begin{proof}
Suppose to the contrary that there is a graph $G\in \mathcal{F}_{n,s,t}$ containing a copy of $H_{s,t}$. Let $k=s+t$ and let $K$ be the copy of $K_{k-1,k-1}$ {(or $3K_3$ when $(s,t)=(3,1)$) embedded in one partite set of } $G$. Then each odd cycle of $H_{s,t}$ must contain odd number of the edges of $K$. Let $A=E(H_{s,t})\cap E(K)$. Then $|A|\geqs k=s+t$. We claim that the center of $H_{s,t}$ must lie in $K$. If not, then $G[A]$ contains a matching of order at least $k$ by the structure of $H_{s,t}$, a contradiction to $\nu(K)=k-1$. Let $x\in V(K)$ be the center of $H_{s,t}$.  Assume that $\deg_{G[A]}(x)=r$.
Let $A_x$ be the set of edges incident with $x$ in $G[A]$.
 Then at most $r$ cycles of $H_{s,t}$ intersect  $A_x$, that is $G[A]-A_x$ contains a matching of $K$ of order at least $k-r$. This is impossible since $\nu(K-N_{G[A]}(x))\leqs k-r-1$.
\end{proof}

In the remaining of the paper, for convenience, we set $\gamma=[400(c(H_{s,t})+1)k]^{-2}$ and $\beta=(c(H_{s,t})+1)\sqrt{\gamma}$, where $c(H_{s,t})$ is the circumference of $H_{s,t}$.

\begin{lem}\label{LEMMA: bipartition with bad vertices}
Let $G$ be a graph on $n$ vertices with $\delta(G)\geqs\lfloor\frac{n}{2}\rfloor$ and
 $e(G)\ge \ex(n, H_{s,t})>e(T_{n,2})$ if $G$ is $H_{s,t}$-free, or  $\phi(G,H_{s,t})\ge \ex(n, H_{s,t})>e(T_{n,2})$, otherwise. Let $V_0\dot\cup V_1$ be a partition of $V(G)$ such that $e(V_0, V_1)$ is maximized and let $m= e(V_0)+e(V_1)$ and $B=\{x\in V(G)\  |\ \deg_{G[V_i]}(x)>\beta n, \mbox{ for $x\in V_i$ and } i=0,1\}$. Then for sufficiently large $n$, the following holds:

 (a) $m<\gamma n^2$ and $|B|<\frac {2\gamma}{\beta}n$;

(b) $\frac{n}{2}-\sqrt\gamma n\leqs|V_i|\leqs\frac{n}{2}+\sqrt\gamma n$ for $i=0,1$;

(c) $e(u,V_{1-i})\geqs\frac{n}{4}-\frac14$ for  $u\in V_i \ (i=0,1)$;

(d) Moreover, $e(u,V_{1-i})\geqs\frac{n}{2}-\beta n-\frac12$ for  $u\in V_i\setminus B\  (i=0,1)$.	
\end{lem}
\begin{proof}
%Note that $e(G)\geqs \ex(G, H_{s,t})$.
Applying  Lemma \ref{LEMMA:partition} and Lemma \ref{LEMMA: stability for phi} to $G$, respectively, with parameter $\gamma$, we have $m<\gamma n^2$ and so $|B|\le \frac{2m}{\beta n}<\frac {2\gamma}{\beta}n$. Let $a=\max\{||V_i|-\frac{n}{2}|,i=0,1\}$. Note that
	$$\lf\frac{n^2}{4}\rf=e(T_{n,2})< e(G)= m+e(V_0,V_1)<\gamma n^2+|V_0||V_1|=\gamma n^2+\frac{n^2}{4}-a^2.$$
Hence we have $a^2\leqs\gamma n^2$ and so $a\leqs\sqrt{\gamma}n$.
	By the choice of $V_0$ and $V_1$, for each  $u\in V_i \, (i=0,1)$ , we have $e(u,V_{1-i})\geqs \deg_{G[V_i]}(u)$. Note that $\deg_G(u)=\deg_{G[V_i]}(u)+e(u, V_{1-i})$ and $\delta(G)\geqs\lf\frac{n}{2}\rf$, we have  $$e(u,V_{1-i})\geqs\max\{\deg_{G[V_i]}(u),\lf\frac{n}{2}\rf-\deg_{G[V_i]}(u)\}.$$
Hence, $e(u,V_{1-i})\geqs\frac12(\deg_{G[V_i]}(u)+\lf\frac{n}{2}\rf-\deg_{G[V_i]}(u))\geqs\frac{n}{4}-\frac14$. Moreover,
	if $u\in V_i\setminus B$, then $e(u,V_{1-i})\geqs\lf\frac{n}{2}\rf-\deg_{G[V_i]}(u)\geqs\frac{n}{2}-\beta n-\frac12$.
\end{proof}

The following technical lemma  will be  used to find copies of $H_{s,t}$ from $G$ in Sections 3 and 4. Let $G$ be a graph and  $V(G)=V_0\dot\cup V_1$ and $x\in V(G)$. We will use $G_i$ and $N_i(x)$ instead of $G[V_i]$ and $N_G(x)\cap V_i$ for $i=0,1$ in the remaining of this paper.

\begin{lem}\label{LEMMA: Main technique}
	Let $s\geqs0, t\geqs1$ and $k=s+t$. Let $G$ be a graph on $n$ vertices satisfying  that

\noindent (i) $V_0\dot{\cup} V_1$ is a partition of $V(G)$ with $\max\{|V_0|,|V_1|\}\leqs(\frac12+\sqrt{\gamma})n$ and, for each $i=0,1$, $V_i$ has a subset $B_i$ with $E(G[B_i])=\emptyset$ and $|B_0\cup B_1|<\sqrt{\gamma} n$ and
  $$|N_{1-i}(u)|>\left\{
	\begin{aligned}
	&2n/5 & &\mbox{if $u\in V_i\setminus B_i$}\\
	&n/9 & &\mbox{if $u\in B_i$}\\
	\end{aligned} \right.,$$

\noindent(ii) $V_i$ has a subset $U_i$ with $|V_i\setminus U_i|<\sqrt{\gamma} n$ for $i=0,1$.

{If there exist a vertex $x\in B_i$ with $\deg_{G_i}(x)\geqs k$, or}
a vertex $x\in V_i\setminus B_i$  and a matching $M_{1-i}$ of $G[N_{1-i}(x)\setminus B_{1-i}]$ with
$\deg_{G_i}(x)+|M_{1-i}|\geqs s$ and
	\begin{equation}\label{e1}
  \deg_{G_i}(x)+|M_{1-i}|+\nu(G[V_i\setminus N_i(x)])+\nu(E_{V_{1-i}\setminus V(M_{1-i})}(x))\geqs k,
\end{equation}
	then for sufficiently large $n$, there is a copy of $H_{s,t}$, say $H$, in $G$ centered at $x$, satisfying that

(1) $H$ contains exactly $k$ edges in $E(V_0)\cup E(V_1)$,

(2) $V(H)\cap U_{1-i}\neq\emptyset$,

(3) if $\deg_{G_i}(x)\geqs k$, then $V(H)\cap U_j\neq\emptyset$ for $j=0,1$.
\end{lem}
\begin{proof}
	 Without loss of generality, assume that there is such a vertex $x\in V_0$. Let $N_0(x)=\{x_1,\cdots,x_\ell\}$ and $M_1=\{w_1z_1,\cdots,w_mz_m\}$ be a matching of $G[N_1(x)\setminus B_1]$ with $\ell+m\geqs s$ and $$\ell+m+\nu(G[V_0\setminus N_0(x)])+\nu(E_{V_1\setminus V(M_1)}(x))\geqs k.$$
Let $\{u_1v_1,\cdots,u_pv_p\}$ and $\{w'_1z'_1,\cdots,w'_qz'_q\}$ be two matchings of $G[V_0\setminus N_0(x)]$ and $E_{V_1\setminus V(M_1)}(x)$ respectively, such that $$\ell+m+p+q=k,$$ and assume that $\{w'_1, \cdots, w'_q\}\subseteq N_1(x)$. { In the case that $x\in B_0$ and $\deg_{G_0}(x)\geqs k$, we simply set $\ell=k$ and $m=p=q=0$.}
Suppose $H_{s,t}$ consists of $k=s+t$ odd cycles of lengths $q_1,q_2,\cdots,q_k$ respectively, where $3=q_1=\cdots=q_s< q_{s+1}\leqs\cdots\leqs q_k$.
	
 Note that $xw_iz_i$ is a triangle for every edge $w_iz_i\in M_1$, $i=1,\cdots,m$. Since $M_1$ is a matching, by using $w_1z_1,\cdots,w_{\min\{m,s\}}z_{\min\{m,s\}}$, we can easily find a copy of $H_{\min\{m,s\},0}$. Next we construct cycles of length $q_{\min\{m,s\}+1},\cdots,q_k$ step by step such that these $k$ odd cycles form an $H_{s,t}$. In another words, we will show that at step $j$ with $\min\{m,s\}<j\leqs k$, we construct a cycle of length $q_j$ which intersects previous constructed cycles only at $x$.
		
\medskip	
\noindent{\bf Case 1.} $j\leqs m$.
\medskip

\noindent	Since $m\geqs j>\min\{m,s\}$, we have $m>s$. It yields that $j>\min\{m,s\}=s$ and $q_j\geqs 5$. Avoiding all vertices except $x$ that have been previously used, we find vertices $w_{0}^1,w_{1}^2,\cdots,w_{0}^{q_j-4}, w_{1}^{q_j-3}$ from $U_0\setminus B_0$ and $U_1\setminus B_1$ alternatively with $w_0^1\in U_0\setminus B_0$ such that $P=z_{j}w_{0}^1 w_{1}^2\cdots w_{1}^{q_j-3}x$ is a path of length $q_j-2$. This is possible since each vertex $u\in V_i$ has at least
$$e_G(u, V_{1-i})-|B_{1-i}|-|V_{1-i}\setminus U_{1-i}|\geqs(\frac19-2\sqrt\gamma)n$$ neighbors in $U_{1-i}\setminus B_{1-i}$,
	and  $w_{0}^{q_j-4}\in U_0\setminus B_0$  and $x$ have at least
	$$ e_{G}(w_{0}^{q_j-4},V_{1})+e_{G}(x,V_{1})-|V_{1}|-|V_1\setminus U_1|-|B_1| >\frac25n+\frac19n-\frac12n-3\sqrt{\gamma}n>\frac{n}{400}$$
common neighbors in $U_{1}\setminus B_1$.
Hence $P\cup \{xw_{j}z_{j}\}$ is a desired $q_j$-cycle.

\medskip	
\noindent{\bf Case 2.} $m<j\leqs \ell+m$.

\medskip
\noindent If $m<j\leqs s$, then $q_j=3$. Note that $E(G[B_i])=\emptyset$ for $i=0,1$. Then at least one of $\{x, x_{j-m}\}$ is not in $B_0$. Hence the number of common neighbors of $x$ and $x_{j-m}$ in $U_1\setminus B_1$ is at least
	$$e_G(x,V_1)+e_G(x_{j-m},V_1)-|V_1|-|V_1\setminus U_1|-|B_1|>\frac25n+\frac19n-\frac12n-3\sqrt{\gamma}n>\frac{n}{400}.$$
So we can find a triangle using $x, x_{j-m}$ and a common neighbor of them which avoids all vertices that have been previously used.

	If $s<j\leqs \ell+m$, then $q_j\geqs5$. For the same reason as in Case 1, we can find an alternating path $P=xw_{1}^1 w_{0}^2 \cdots w_1^{q_j-4}w_{0}^{q_j-3}$ with vertices chosen from $U_0\setminus B_0$ and $U_1\setminus B_1$ alternatively with $w_1^1\in U_1\setminus B_1$  and a common neighbor of $x_{j-m}$ and $w_0^{q_j-3}$ in $V_1\setminus B_1$, say $w_1^{q_j-2}$, avoiding all vertices except $x$ that have been previously used. Hence $P\cup\{w_{1}^{q_j-2}x_{j-m}x\}$ is a desired $q_j$-cycle.

\medskip	
\noindent{\bf Case 3.} $\ell+m<j\leqs \ell+m+p$.
	
\medskip
\noindent {Since $j>\ell+m\geqs s$ and $p>0$, we have $q_j\geqs5$ and $x\notin B_0$.} Since $E(G[B_0])=\emptyset$, at least one of $\{u_{j-\ell-m}, v_{j-\ell-m}\}$ is not in $B_0$. Assume that $u_{j-\ell-m}\notin B_0$. Hence, with the same reason as in Case 1 and Case 2, avoiding all vertices except $x$ that have been previously used, we first find a common neighbor of $x$ and $u_{j-\ell-m}$, say $w_1^1$, in $U_1\setminus B_1$, next find an alternating path  $P=xw_{1}^3 w_{0}^4 \cdots w_{0}^{q_j-3}$
%\footnote{WHY is true for $q_j=5$? }
of length $q_j-4$ with vertices chosen from $U_0\setminus B_0$ and $U_1\setminus B_1$, alternatively (where if $q_j=5$, we set $P=x=w_{0}^{q_j-3}$), and finally, find a common neighbor of $w_{0}^{q_j-3}(\notin B_0)$ and $v_{j-\ell-m}$, say $w_{1}^{q_j-2}$ in $U_1\setminus B_1$. Therefore,  $\{u_{j-\ell-m}w_1^1x\}\cup P\cup\{w_{1}^{q_j-2}v_{j-\ell-m}u_{j-\ell-m}\}$ is a desired $q_j$-cycle.

\medskip	
\noindent{\bf Case 4.} $\ell+m+p<j\leqs k$.

\medskip
\noindent
	For the same reason as the above, avoiding all vertices except $x$ that have been previously used, we first find an alternating path $P=z'_{j-\ell-m-p}w_{0}^1w_{1}^2\cdots w_{0}^{q_j-4}$ of length $q_j-3$ with vertices chosen from $U_0\setminus B_0$ and $U_1\setminus B_1$ alternatively, next a common neighbor of $w_{0}^{q_j-4}(\notin B_0)$ and $x$, say $w_{1}^{q_j-3}$,  in $U_1\setminus B_1$. Then $P\cup\{w_0^{q_j-4}w_{1}^{q_j-3}xw'_{j-\ell-m-p}z'_{j-\ell-m-p}\}$ is a desired $q_j$-cycle.
	
	Thus we always can find a copy of $H_{s,t}$ centered at $x$. In each step, the new constructed $q_j$-cycle uses exactly one edge in $E(V_0)\cup E(V_1)$ and at least one new vertex in $U_1$, so (1) and (2) hold. Moreover, if $\deg_{G_i}(x)\geqs k$, then we choose $k$ neighbors of $x$ in $G_i$ and set $\ell=k$ and $p=m=q=0$. Hence we find a copy of $H_{s,t}$ only using Case 2. In Case 2, if $q_j\geqs5$, then the $q_j$-cycle we found uses at least one vertex in $U_1$ and a vertex in $U_0$. So (3) holds.
\end{proof}

\section{Proof of Theorem \ref{THM:MAIN THEOREM 1}}
We begin with a technical lemma, which is crucial to the proof of Theorem \ref{THM:MAIN THEOREM 1} and also will be used in next section.
\begin{lem}\label{LEMMA: Force to extremal graph}
Given integers $s\geqs0, t\geqs1 $ and $H_{s,t}\in\mathcal C_{s,t}$.  Let $k=s+t$ and let $G$ be a graph on $n$ vertices with $V(G)=V_0\dot\cup V_1$ and $e(G)\geqs e(T_{n,2})+(k-1)^2$. If $G$ satisfies \\
(i) $||V_i|-\frac n2|\leqs {\sqrt{\gamma}n}$ and $e(u,V_{1-i})>\frac n2-c$ for each $u\in V_{i}$ and $i=0,1$, where $c$ is a constant, and \\
(ii) for any vertex $x\in V_i,~(i=0,1)$ and any maximum matching $M_{1-i}$ of $G[N_{1-i}(x)]$  with $\deg_{G_i}(x)+|M_{1-i}|\geqs s$, we have
\begin{equation}\label{e2}
 \deg_{G_i}(x)+|M_{1-i}|+\nu(G[V_i\setminus N_i(x)])+\nu(E_{V_{1-i}\setminus V(M_{1-i})}(x))\leqs k-1,
\end{equation}
then for all sufficiently large $n$, $e(G)=e(T_{n,2})+(k-1)^2$. Moreover, if $G$ is $H_{s,t}$-free, then $G\in \mathcal F_{n,s,t}$.
\end{lem}
\noindent{\it Proof.} %For convience, we denote $G_i=G[V_i],\nu_i=\nu(G_i)$ and $\Delta_i=\Delta(G_i)$, for $i=0,1$.
Let $m=e(V_0)+e(V_1)$. Since $m+e(V_0,V_1)= e(G)\geqs e(T_{n,2})+(k-1)^2$, we have $m\geqs(k-1)^2$, with equality holds only if $G$ contains a balanced complete bipartite subgraph with partitions $V_0$ and $V_1$. Condition (ii) implies that
\begin{claim}\label{Claim: Delta<k}
$\max\{\Delta_0,\Delta_1\}\leqs k-1$.
\end{claim}
% Firstly we give following Claim.
Condition (ii) also implies that $\nu_i\le k-1$ for $i=0,1$. Furthermore, we have
\begin{claim}\label{Claim:nu0+nu1<k}
$\nu_0+\nu_1\leqs k-1.$
\end{claim}
\noindent%{\it Proof.}
	We prove it by contradiction. Suppose that $\nu_0+\nu_1\geqs k$. Let $F_0$ and $F_1$ be two maximum matchings of $G_0$ and $G_1$, respectively. Then  $|F_0|+|F_1|\ge k$. Let $A_i=\cap_{v\in V(F_{1-i})}N_i(v)$. We first show that
	%\begin{case}\label{FACT:A_i not empty and 2leqsNu}

\medskip
\noindent{\bf Claim 2.1}\label{FACT:A_i not empty and 2leqsNu}  $A_i\neq\emptyset$ and $2\leqs\nu_i\leqs s-1\leqs k-2$ for every $i=0,1$.\\
	%\end{case}
 For each $i=0,1$, since $||V_i|-\frac n2|\leqs {\sqrt{\gamma}n}$ and $e(u,V_i)>\frac n2-c$ for all $u\in V_{1-i}$.
 %(We may assume $\nu_{1-i}\leqs k-1$, otherwise we can pick a corresponding matching of size $k$ and use the same method as follows to get a contradiction.)
 By definition of $A_i$,
		\begin{eqnarray*}
			|A_i|&\geqs& {2|F_{1-i}|(\frac n2-c)-(2|F_{1-i}|-1)|V_i|}\\
			&\geqs& {2|F_{1-i}|(\frac n2-c)-(2|F_{1-i}|-1)(\frac n2+{\sqrt{\gamma}n})}\\
			& >&{\frac n2-2|F_{1-i}|(c+{\sqrt{\gamma}n})}\\
				&>&{(\frac12-2k\sqrt{\gamma})n-2kc},
		\end{eqnarray*}
the last inequality holds since $|F_{1-i}|= \nu_{1-i}\le k-1$.
So $A_i\neq\emptyset$ for sufficiently large $n$, and furthermore, for any vertex $x\in A_i$, $|M_{1-i}|=\nu(G[N_{1-i}(x)])=\nu_{1-i}$. It is easy to show  that $\deg_{G_i}(x)+\nu(G[V_i\setminus N_i(x)])\geqs|F_i|=\nu_i$.
% and $|M_{1-i}|+\nu(E_{V_{1-i}\setminus V(M_{1-i})}(x))\ge \nu_{1-i}$.
Hence
 %Note that $M_{1-i}\in \mathcal M^2_{V_{1-i}}(x)$ and $$e(x,V_i)+\nu_{1-i}+\nu(G[V_i\setminus N(x)])+\nu(E_{V_{1-i}\setminus V(M_{1-i})}(x))\geqs\nu_i+\nu_{1-i}\geqs k.$$
$$\deg_{G_i}(x)+|M_{1-i}|+\nu(G[V_i\setminus N_i(x)])+\nu(E_{V_{1-i}\setminus V(M_{1-i})}(x))\geqs\nu_i+\nu_{1-i}\geqs k.$$
Thus we must have $\deg_{G_i}(x)+|M_{1-i}|\leqs s-1$, otherwise we have a contradiction to condition (ii). So $|M_{1-i}|=\nu_{1-i}\leqs s-1=k-t-1\leqs k-2$ for $i=0, 1$ (since $t\geqs1$). Therefore, $s-1+\nu_i\geqs\nu_{1-i}+\nu_i\geqs k$ and thus $\nu_i\geqs k-s+1=t+1\geqs 2$ .	
%	\end{proof}

%\begin{fact}\label{FACT:e(G)< leads to contra}
\medskip
\noindent{\bf Claim 2.2}\label{FACT:e(G)< leads to contra}
	We have that $e(G)<e(T_{n,2})+(k-1)^2.$  So we get a contradiction to $e(G)\ge e(T_{n,2})+(k-1)^2$.

\medskip
%\end{fact}	
%\begin{proof}
	%Delete vertices of $G$ so that $e(G)-|V_0||V_1|$ is maximal, and let
	%$G$ be minimal in this sense.
%	We now claim that for each $i=0,1$ and every $x\in V_{1-i}$.
%	$$\deg_G(x)-|V_{1-i}|>0.$$
%	In fact, if for some $x\in V_i$ such $\deg_G(x)\leqs|V_{1-i}|$ holds, then
%	\begin{eqnarray*}
%	e(G-x)-|V_i\setminus x||V_{1-i}|&=&(e(G)-\deg_G(x))-(|V_0||V_1|-|V_{1-i}|)\\
%	&=&e(G)-|V_0||V_1|+(|V_{1-i}|-\deg_G(x))\\
%	&\geqs& e(G)-|V_0||V_1|,
%	\end{eqnarray*}
%	contradicting the minimality of $G$.
\noindent	For any $x\in V_i$ ($i=0,1$),
%	$$\deg_G(x)-|V_{1-i}|\leqs k-1-\nu_{1-i}.$$
	Condition (ii) implies that $\deg_{G_i}(x)+|M_{1-i}|\leqs k-1$. Hence,
	\begin{eqnarray*}
		\deg_G(x)-|V_{1-i}|&=&\deg_{G_i}(x)+e_G(x,V_{1-i})-|V_{1-i}|\\
		&\leqs&k-1-(|M_{1-i}|+|V_{1-i}|-e_G(x,V_{1-i}))\\
		&\leqs&k-1-\nu_{1-i},
	\end{eqnarray*}	
where the last inequality holds since any maximum matching of $G_{1-i}$ intersects $G[N_{1-i}(x)]$ at most $|M_{1-i}|$ edges and intersects $G_{1-i}-E(G[N_{1-i}(x)])$ at most $|V_{1-i}|-e_G(x,V_{1-i})$ edges.
%where the reason why the last inequality holds is as follows. Since $M_{1-i}\setminus(M_{1-i}\cap E(G[N(x)\cap V_{1-i}]))$ is a matching that each of edge has at least one endpoint in $V_{1-i}\setminus N(x)$, we have that $|M_{1-i}\setminus(M_{1-i}\cap E(G[N(x)\cap V_{1-i}]))|\leqs |V_{1-i}|-e_G(x,V_{1-i})$. Also note that $M_{1-i}\cap E(G[N(x)\cap V_{1-i}]$ is a matching in $G[N(x)\cap V_{1-i}]$, we have that $|M_{1-i}\cap E(G[N(x)\cap V_{1-i}])|\leqs \nu(G[N(x)\cap V_{1-i}])$. Hence put these two together we can conclude that $\nu_{1-i}=|M_{1-i}|\leqs|V_{1-i}|-e_G(x,V_{1-i})+\nu(G[N(x)\cap V_{1-i}]).$
Now apply Lemma~\ref{LEMMA: Sum min deg b leqs} to $G_i$ $(i = 0, 1)$ with $\Delta_i\le k-1, \nu_{i}$ and $b= \Delta_i-\nu_{1-i}$ ($\leqs \Delta_i-2$ by Claim 2.1), we get
\begin{eqnarray*}
\sum_{x\in V_i}(\deg_G(x)-|V_{1-i}|) & \le & \sum_{x\in V_i}\deg_{G_i}(x)\\
 &\leqs &\sum_{x\in V_i}\min\{\deg_{G_i}(x), \Delta_i-\nu_{1-i}\}\\
 &\leqs &\nu_i(2\Delta_i-\nu_{1-i})\le \nu_i(2k-2-\nu_{1-i}) .
%\deg_G(x)-|V_{1-i}|&=&\deg_{G_i}(x)+e_G(x,V_{1-i})-|V_{1-i}|\\
%		&\leqs&k-1-(|M_{1-i}|+|V_{1-i}|-e_G(x,V_{1-i}))\\
%		&\leqs&k-1-\nu_{1-i},
	\end{eqnarray*}
%$$\sum_{x\in V_i}(\deg_G(x)-|V_{1-i}|)\leqs\sum_{x\in V_i}\min\{e(x,V_i),k-1-\nu_{1-i}\}\leqs\nu_i(2k-2-\nu_{1-i}).$$
Summing over $i$ for $i=0,1$, we have
\begin{eqnarray*}
2e(G)-2|V_0||V_1|&\leqs&2\big[(k-1)(\nu_0+\nu_1)-\nu_0\nu_1\big]\\
&=&2\big[(k-1)^2-(k-1-\nu_0)(k-1-\nu_1)\big]\\
&<&2(k-1)^2,
\end{eqnarray*}
the last inequality holds since $\nu_i\leqs k-2$ by Claim 2.1.
Hence, $e(G)<|V_0||V_1|+(k-1)^2\leqs e(T_{n,2})+(k-1)^2$. This completes the proof of Claim 2.2 and so of Claim~\ref{Claim:nu0+nu1<k}.

%\medskip
%Based on Claim \ref{Claim:nu0+nu1<k}, we can get more delicate structure of $G$ by following claims.

\begin{claim}\label{CLAIM:Dta and Nu attaining max}
	If $\max\{\Dta_0,\Dta_1\}\leqs k-2$, then $k=4$ and $e(G)=e(T_{n,2})+(k-1)^2$. Furthermore, if $G$ is $H_{s,t}$-free, then $(s,t)=(3,1)$ and $G\in \mathcal{F}_{n,3,1}$.
\end{claim}
%\begin{proof} %Suppose that $\max\{\Dta_0,\Dta_1\}\leqs k-2$.
\noindent By Lemma \ref{LEMMA:f_Nu_Delta},
	\begin{eqnarray*}
		m&=&e(V_0)+e(V_1)\leqs f(\nu_0,\Delta_0)+f(\nu_1, \Delta_1)\\
		&\leqs&f(\nu_0+\nu_1,k-2)\leqs f(k-1,k-2).
	\end{eqnarray*}
If $k\neq 4$, then $m\leqs f(k-1,k-2)=(k-1)^2-1$, a contradiction to $m\geqs (k-1)^2$. So $k=s+t=4$ and we have $m\leqs f(3,2)=(k-1)^2=9$. Therefore, $m=(k-1)^2=9$ and so $G$ contains a complete balanced bipartite subgraph with partite sets $V_0$ and $V_1$. Let $H$ be the subgraph consisting of nonempty components of $G_0\cup G_1$.  Then $H$ is a graph with $e(H)=9$, $\Delta(H)=2$ and $\nu(H)=3$. Hence $|V(H)|\ge e(H)=9$, the equality holds if and only if $H$ is 2-regular. By Observation~\ref{OBS: o1} and the fact that $\nu(H)\ge \omega(H)$, we have $|V(H)|=9$ and $\omega(H)=3$, that is $H=3K_3$. By Lemma~\ref{LEMMA:MAIN LEMMA}, $s=3$ and $t=1$.
%In this case, $G$ is a Tur\'an graph $T_{n,2}$ with three vertex-disjoint triangles embeded into its classes.
Then $G$ must be a Tur\'an graph $T_{n,2}$ with $H$ embedded into one class (that is $G\in\mathcal{F}_{n,3,1}$), otherwise we can easily find a vertex which contradicts condition (ii) of the lemma.
%$H_{s,t}$ in $G$ by applying Lemma~\ref{LEMMA: Main technique} (choose $B_i=\emptyset$ and $U_i=V_i$ for $i=0,1$).
%\end{proof}

\begin{claim}\label{CLAIM:e0e1=0} If $\max\{\Dta_0,\Dta_1\}=k-1$, then  $e(V_0)\cdot e(V_1)=0$.
\end{claim}
%\begin{proof}
\noindent By Lemma \ref{LEMMA:f_Nu_Delta} and Claim \ref{Claim:nu0+nu1<k}, we have
	\begin{eqnarray*}
		m&\leqs& f(\nu_0,k-1)+f(\nu_1,k-1)\leqs f(\nu_0+\nu_1,k-1)\\
		&\leqs& f(k-1,k-1)\leqs k(k-1).
	\end{eqnarray*}	
	Wlog, we assume $\Dta_0=k-1$. Let $x\in V_0$ with $\deg_{G_0}(x)=k-1$. We will show that $e(V_1)=0$.
	If not, then $\nu_1\geqs1$ and so $\nu_0\leqs k-2$. Let $A_1=\{u\in{V_1}:\deg_{G_1}(u)>0\}$.
	By~(\ref{e2}), we have $A_1\cap N_1(x)=\emptyset$.
	So $e(V_0,V_1)\leqs|V_0||V_1|-|A_1|\leqs e(T_{n,2})-|A_1|$. Thus we have  $$e(T_{n,2})+(k-1)^2\leqs e(G)\leqs e(T_{n,2})-|A_1|+m.$$
	Therefore,  $|A_1|\leqs m-(k-1)^2$. That is $|A_1|+(k-1)^2\leqs m$.  Again by Lemma~\ref{LEMMA:f_Nu_Delta},  we have
	\begin{eqnarray*}
		m&\leqs&f(\nu_0,\Dta_0)+f(\nu_1,\Dta_1)
		\leqs \nu_0(\Dta_0+1)+\nu_1(\Dta_1+1)\\
		&\leqs&\nu_0k+(k-1-\nu_0)(\Dta_1+1)\, \mbox{ (since $\Delta_0=k-1$ and $\nu_0+\nu_1\leqs k-1$)}\\
		&\leqs&\nu_0k+(k-1-\nu_0)|A_1|\, \mbox{ (since $\Dta_1+1\leqs|A_1|$)}\\
		&=&\nu_0(k-|A_1|)+(k-1)|A_1|\\
		&\leqs&(k-2)(k-|A_1|)+(k-1)|A_1|\, \mbox{ (since $|A_1|\leqs k-1$ and $\nu_0\leqs k-2$)}\\
		& =&(k-1)^2+|A_1|-1\\
       %\leqs (k-1)^2+m-(k-1)^2-1\\
		&\le  & m-1, \mbox{ a contradiction.}
	\end{eqnarray*}
%\end{proof}

 Claim \ref{CLAIM:Dta and Nu attaining max} and Claim \ref{CLAIM:e0e1=0} implies that $\max\{\Dta_0,\Dta_1\}=k-1$ and $e(V_0)\cdot e(V_1)=0$. Without loss of generality, assume that $e(V_1)=0$. Then $m=e(V_0)$,  $\Delta_0=k-1$ and $\Delta_1=0$. Let $A_0$ be the set of non-isolated vertices in $G_0$. Claim~\ref{Claim:nu0+nu1<k} implies that $\nu(G[A_0])\le k-1$. Condition (ii) and Lemma \ref{LEMMA:MAIN LEMMA} imply that $m=e(G[A_0])\leqs (k-1)^2$. Thus we have $m=(k-1)^2$ and $G$ contains a complete balanced bipartite subgraph with classes $V_0$ and $V_1$. Again by Lemma \ref{LEMMA:MAIN LEMMA}, $G[A_0]\cong K_{k-1,k-1}$. That is $G\in\mathcal{F}_{n,s,t}$. This completes the proof of Lemma \ref{LEMMA: Force to extremal graph}.
\qed\\

\noindent{\it Proof of Theorem \ref{THM:MAIN THEOREM 1}}. Let $G$ be an extremal graph on $n$ vertices for $H_{s,t}$, where $n$ is large enough. Let $k=s+t$. By Lemma~\ref{LEMMA:Gnk is Ckq-free}, we have $e(G)\geqs e(T_{n,2})+(k-1)^2$. By Lemma~\ref{Lemma:delete to get large minidegree}, we may assume that $\delta(G)\geqs\lf\frac{n}{2}\rf$. Let $E(V_0, V_1)$ be a maximum cut of $G$ and let $B$ be defined as in Lemma \ref{LEMMA: bipartition with bad vertices}.
By Lemma~\ref{LEMMA: bipartition with bad vertices}, we have

(a) $m=e(V_0)+e(V_1)<\gamma n^2$ and $|B|<\frac 12{\sqrt\gamma}n<\frac 18\beta n$;

(b) $\frac{n}{2}-\sqrt\gamma n\leqs|V_i|\leqs\frac{n}{2}+\sqrt\gamma n$ for $i=0,1$;

(c) $e(u,V_{1-i})\geqs\frac{n}{4}-\frac14$ for  $u\in V_i \ (i=0,1)$;

(d) Moreover, $e(u,V_{1-i})\geqs\frac{n}{2}-\beta n-\frac12$ for  $u\in V_i\setminus B\  (i=0,1)$.

{ Let $B_i=B\cap V_i$ for $i=0,1$.} Then,  for each $v\in B_i,\, (i=0,1)$, $e_{G_i}(v, V_i\setminus B_i)\ge k\lc\frac{1}{2k}\deg_{G_i}(v)\rc$ since %the number of bad vertices is at most
$$|B_i|\le|B|\leqs\frac 18\beta n\leqs\frac{1}{2}\beta n-k\leqs \deg_{G_i}(v)-k\lc\frac{1}{2k}\deg_{G_i}(v)\rc.$$
Hence we can keep $k\lc\frac{1}{2k}\deg_{G_i}(v)\rc$ edges of $E(v, V_i\setminus B_i)$ and delete the other edges incident with $v$ in $G_i$ for each  $v\in B_i$.
%This is possible since %the number of bad vertices is at most
%$$|B|\le \frac{2m}{\beta n}\le \frac{2\gamma}{\beta}n\leqs\frac{1}{2}\beta n-k\leqs \deg_{G_i}(v)-k\lc\frac{1}{2k}\deg_{G_i}(v)\rc.$$
%We repeat this procedure for each vertex in $B$.
Denote the resulting graph by $G'$.

{Note that $E(G'[B_i])=\emptyset$ and condition (i) of Lemma~\ref{LEMMA: Main technique} is guaranteed by (b) and (c). Since $G'$ is $H_{s,t}$-free, $G'$ has no vertex in $B_i$ ($i=0,1$), say $v$, satisfying that $\deg_{G'[V_i]}(v)\geqs k$, otherwise we can find a copy of $H_{s,t}$, by applying Lemma~\ref{LEMMA: Main technique} to $G'$ with $B_i$ and $U_i=V_i$ ($i=0,1$). So we must have $B_i=\emptyset~(i=0,1)$ since $k-1<\frac12\beta n\leqs k\lc\frac{1}{2k}\deg_{G_i}(v)\rc=\deg_{G'[V_i]}(v)$ for every $v\in B_i~ (i=0,1)$ and thus $G=G'$.

Now note that $\max\{\Delta_0, \Delta_1\}\leqs k-1$ and $\delta(G)\geqs\lf\frac{n}{2}\rf$. Together with (b), condition (i) of Lemma~\ref{LEMMA: Force to extremal graph} is guaranteed. Meanwhile, by Lemma \ref{LEMMA: Main technique}, since $G$ is $H_{s,t}$-free, all vertices of $V(G)$ do not satisfy inequality~(\ref{e1}) (or equivalently, for any vertex $x\in V_i,~(i=0,1)$ and any maximum matching $M_{1-i}$ of $G[N_{1-i}(x)]$  with $\deg_{G_i}(x)+|M_{1-i}|\geqs s$, $x$ satisfies inequality~(\ref{e2})). So condition $(ii)$ of Lemma \ref{LEMMA: Force to extremal graph} is also guaranteed. Thus, apply Lemma \ref{LEMMA: Force to extremal graph} to $G$, since $G$ is $H_{s,t}$-free, we have $G\in\mathcal F_{n,s,t}$.
}
%Since $\delta(G)\geqs\lf\frac{n}{2}\rf$, we have $e_G(u, V_{1-i})\geqs\lf\frac{n}{2}\rf-(k-1)$ for each vertex $u\in V_i$, and  $\lf\frac{n}{2}\rf-(k-1)\leqs|V_i|\leqs\lc\frac{n}{2}\rc+(k-1)$, $i=0,1$. Now note that since $G$ is $H_{s,t}$-free, all other conditions in Lemma \ref{LEMMA: Force to extremal graph} can be guaranteed by Lemma \ref{LEMMA: Main technique}, thus $G\in\mathcal F_{n,s,t}$.
\qed

\section{Proof of Theorem \ref{THM:MAIN THEOREM 2}}
We need to show that provided with large enough integer $n$, $\phi(G,H_{s,t})\leqs\ex(n,H_{s,t})$ for all graphs $G$ on $n$ vertices, with equality holds if and only if $G\in \mathcal F_{n,s,t}$. Since $\phi(G,H_{s,t})=e(G)-p_{H_{s,t}}(G)(e(H_{s,t})-1)$, it suffices to show that  $$p_{H_{s,t}}(G)\geqs\frac{e(G)-\ex(n,H_{s,t})}{e(H_{s,t})-1}$$  for all graphs $G$ on $n$ vertices with $e(G)\geqs\phi(G,H)\geqs \ex(n,H_{s,t})=e(T_{n,2})+(k-1)^2$ (by Theorem \ref{THM:MAIN THEOREM 1}), with equality holds if and only if $G\in\mathcal F_{n,s,t}$.

Let $G$ be such a graph. By Lemma \ref{Lemma:phi case delete to get large minidegree}, we may assume without loss of generality that $\delta(G)\geqs\lf\frac{n}{2}\rf$. Let $E(V_0,V_1)$ be a maximum cut of $G$. Let $B$ be defined as in Lemma~\ref{LEMMA: bipartition with bad vertices} and $B_i=B\cap V_i$ for $i=0,1$. Then, by Lemma~\ref{LEMMA: bipartition with bad vertices}, we have

(a) $m=e(V_0)+e(V_1)<\gamma n^2$ and $|B|<\frac {2\gamma}{\beta}n$;

(b) $\frac{n}{2}-\sqrt\gamma n\leqs|V_i|\leqs\frac{n}{2}+\sqrt\gamma n$ for $i=0,1$;

(c) $e(u,V_{1-i})\geqs\frac{n}{4}-\frac14$ for  $u\in V_i \ (i=0,1)$;

(d) Moreover, $e(u,V_{1-i})\geqs\frac{n}{2}-\beta n-\frac12$ for  $u\in V_i\setminus B_i\  (i=0,1)$.

Let $$C(H_{s,t})=\frac{2k(k-1)(2e(H_{s,t})-k-1)}{e(H_{s,t})-2k-1}$$ be a constant that depends only on $H_{s,t}$. We divide the proof into two cases according to $m>c(H_{s,t})$ or $m\leqs c(H_{s,t})$.

\medskip
\noindent{\bf Case 1.}	$m>C(H_{s,t})$.

\medskip
%\noindent Then $p_{H_{s,t}}(G)>\frac{e(G)-\ex(n,H_{s,t})}{e(H_{s,t})-1}$.

\noindent We do the same operation to $G$  as in the proof of Theorem~\ref{THM:MAIN THEOREM 1}.  That is we keep $k\lc\frac{1}{2k}\deg_{G_i}(v)\rc$ edges of $G$ that connect $v$ to its neighbors in $V_i\setminus B_i$ and delete the other edges incident with $v$ in $G_i$ for each vertex $v\in B_i,\, i=0,1$. Denote the the resulting graph by $G^0$. %First we claim that give following fact.
%\begin{fact}\label{Fact:geqs m/2}

\medskip

\noindent{\bf Claim 1.}	We have that $e(G^0[V_0])+e(G^0[V_1])\geqs\frac{m}{2}.$
%\end{fact}
%\begin{proof}

\medskip
\noindent	For each $i\in \{0,1\}$, we can see
	\begin{eqnarray*}
		e(G^0[V_i]) &=& e(G^0[V_i\setminus B_i])+\sum_{v\in B_i}\deg_{G^0[V_i]}(v)\\
		&=&e(G[V_i\setminus B_i])+\sum_{v\in B_i}k\lc\frac{1}{2k}\deg_{G_i}(v)\rc\\
		&\geqs&\frac12e(G[V_i\setminus B_i])+\frac12\sum_{v\in B_i}\deg_{G_i}(v)\\
		&\geqs&\frac12e(V_i).
	\end{eqnarray*}
Hence, $e(G^0[V_0])+e(G^0[V_1])\geqs \frac 12 e(V_0)+\frac 12 e(V_1)\ge \frac{m}{2}$.
%\end{proof}

\medskip

 We will use the following algorithm to find enough edge-disjoint copies of $H_{s,t}$ in $G$. Initially, {set $U_i^0=V_i\setminus B_i$, for $i=0,1$.}

\medskip
\noindent{\bf Algorithm 1}: Begin with $G^0, U_0^0,U_1^0,B_0,B_1$, suppose that we have gotten $G^j$ and $U_0^j, U_1^j$ for some $j\geqs0$. A vertex $u\in V_i\setminus B_i$ $(i=0,1)$ is {\it active} in $G^j$ if $e_{G^j}(u, V_{1-i})\geqs\frac{n}{2}-2\beta n-\frac{1}{4}$, otherwise, is {\it inactive}. It's clear that all vertices in $V(G^0)\setminus B$ are active in $G^0$ by (d) of Lemma~\ref{LEMMA: bipartition with bad vertices}.

\medskip
\noindent{\bf Step 1.} If there is some $u\in V_i$ with $\deg_{G^j[V_i]}(u)\geqs k$ (here the vertices of $B$ are considered first), applying Lemma \ref{LEMMA: Main technique} to $G^j$ with $V_i, U_i^j, B_i$ for $i=0,1$,  then we can find a copy of $H_{s,t}$ in $G^j$. Let $G^{j+1}$ be the graph obtained from $G^{j}$ by deleting the edges of the $H_{s,t}$.  {Let  $U_i^{j+1}$ be the set of all active vertices of $V_i$ for $i=0,1$.} We stop at some iteration $G^a$ and turn to Step 2 if there is no vertex $u\in V_i$ with $\deg_{G^a[V_i]}(u)\ge k$.

\medskip
\noindent{\bf Step 2.} If there is a matching of size $k$ in $G^j[V_i]$ for some $i\in\{0,1\}$, then, again applying Lemma~\ref{LEMMA: Main technique} to $G^j$, we can find a copy of $H_{s,t}$ in $G^j$. Update $G^j$ to $G^{j+1}$ by the same method as in Step 1. If there there is no such a matching, we stop and denote the resulting graph by $G'$.

Clearly, $\Delta(G'[V_i])\le k-1$ and $\nu(G'[V_i])\}\leqs k-1$ for $i\in\{0,1\}$. So by Lemma \ref{LEMMA:f_Nu_Delta},  we have $e(G'[V_i])\leqs f(k-1,k-1)\leqs k(k-1)$ for $i\in\{0,1\}$. Note that in each step, the copy of $H_{s,t}$ we found uses exactly $k$ edges of $E(G^0[V_0])\cup E(G^0[V_1])$. Thus the number of edge-disjoint copies of $H_{s,t}$ we have found after Step 1 and Step 2 finished is equal to
\begin{eqnarray*}
	& &\frac{e(G^0[V_0])+e(G^0[V_1])-e(G'[V_0])-e(G'[V_1])}{k}
\geqs\frac{m/2-2k(k-1)}{k}\\
	&&>\frac{m-(k-1)^2}{e(H_{s,t})-1}=\frac{(e(T_{n,2})+m)-(e(T_{n,2})+(k-1)^2))}{e(H_{s,t})-1}\\
	&&\geqs\frac{e(G)-\ex(n,H_{s,t})}{e(H_{s,t})-1},
\end{eqnarray*}
the first inequality holds by Claim 1 and the second inequality holds since $m>C(H_{s,t})$. So to complete proof of Case 1, it suffices to show that Algorithm 1 can be successfully iterated. We prove it in the following claim.
%\qed

\medskip
%\begin{claim}
\noindent{\bf Claim 2.}	Algorithm 1 can be successfully iterated.
%\end{claim}
%\begin{proof}

\medskip

\noindent	Let $G^j$ ($j\ge 0$) be the graph obtained  at some point of the iteration. It suffices to verify that $G^j$ satisfies the conditions of Lemma \ref{LEMMA: Main technique}. Note that the total number of iterations is most $\frac mk<\frac{\gamma n^2}{k}$.
%, and in each iteration, an active vertex $u\in V_i$ loses at most $\Delta(H_{s,t})=2k$ edges from $E(u, V_{1-i})$.
In each iteration, there are $e(H_{s,t})-k$ edges removed from $E(V_0, V_1)$. Note that $e_{G^0}(u,V_{1-i})\geqs\frac{n}{2}-\beta n-\frac12$ for  $u\in V_i\setminus B_i\  (i=0,1)$ and $e_{G^j}(u, V_{1-i}){\color{red}<} \frac{n}{2}-2\beta n-\frac12$ for each inactive vertex $u\in V_i$ in $G^j$. So the number of inactive vertices in $G^j$ is at most $\frac{(e(H_{s,t})-k)\frac mk}{\beta n}<\frac{(c(H_{s,t})-1)\gamma}{\beta}n$. {(Recall that $c(H_{s,t})$ is the circumference of $H_{s,t}$.) So $|V_i\setminus U^j_i|<\frac{(c(H_{s,t})-1)\gamma}{\beta}n\leqs\sqrt{\gamma}n$.

For vertex $u\in U_i^j$, since $u$ is active, we have
	$e_{G^j}(u,V_{1-i})\geqs\frac{n}{2}-2\beta n-\frac{1}{4}>\frac25n.$
For vertex $u\in B_i$, $u$ was involved in at most $\deg_{G^0[V_i]}(u)/k$ previous iterations and lost $k$ edges from $E_{G_0}(u, V_{1-i})$ in each iteration. Therefore,
	\begin{eqnarray*}
		&&e_{G^j}(u,V_{1-i})\geqs e_{G^0}(u,V_{1-i})-k\cdot e_{G^0}(u,V_i)/k=e_{G}(u,V_{1-i})-e_{G^0}(u,V_i)\\
		&&\geqs e_{G}(u,V_{1-i})-\frac12e_{G}(u,V_i)-k\geqs\frac12e_{G}(u,V_{1-i})-k\geqs\frac{n}{8}-k-1>\frac n9.
	\end{eqnarray*}
For each $u\in V_i^j\setminus(U^j_i\cup B_i)$,  $u$ must become inactive in some previous iteration, say in $G^{j'}$, with $j'\le j$.
}
Note that after $u$ became inactive, $u$ lost at most $\deg_{G^{j'}[V_i]}(u)$ edges from $E_{G^0}(u, V_{1-i})$. Hence, we have
	\begin{eqnarray*}
		&&e_{G^j}(u,V_{1-i})\geqs e_{G^{j'}}(u,V_{1-i})-\deg_{G^{j'}[V_i]}(u)\geqs e_{G^{j'}}(u,V_{1-i})-\deg_{G^{0}[V_i]}(u)\\
		&&>\frac{n}{2}-2\beta n-\frac{1}{4}-\Delta(H_{s,t})-\beta n=\frac{n}{2}-3\beta n-\frac{1}{4}-\Delta(H_{s,t})>\frac25n.
	\end{eqnarray*}
Therefore, Step 1 of Algorithm 1 can be iterated successfully. As for Step 2, suppose that $G^j[V_i]$ has a matching $M_i$ of size $k$. It suffices to show $V(M_i)$ has a common neighbor in $V_{1-i}$. To see this, note that each vertex $u\in B_i$ has degree $k\lceil\frac{\deg_{G_i}(u)}{2k}\rceil$ in $G^0[V_{i}]$; hence in any iteration $G^{j'}$ of Step 1, $\deg_{G^{j'}[V_i]}(u)$ must be a multiple of $k$. So, after Step 1 was finished, $u\in B_i$ must have in-degree zero. Thus we have $V(M_i)\cap B_i=\emptyset$ and so the number of common neighbors of $V(M_{i})$ in $V_{1-i}$ is at least $$2k(\frac{n}{2}-3\beta n-\frac{1}{4}-\Delta(H_{s,t}))-(2k-1)|V_{1-i}|\geqs\frac n2-2k(\sqrt\gamma+3\beta)n\geqs\frac{24}{50}n.$$
%\end{proof}

\medskip

%\begin{lem}
\noindent{\bf Case 2.}	$m\leqs C(H_{s,t})$.
%\end{lem}

\medskip
\noindent %Then $p_{H_{s,t}}(G)\geqs\frac{e(G)-\ex(n,H_{s,t})}{e(H_{s,t})-1}$, with equality holds if and only if $G\in\mathcal F_{n,s,t}$.
Since $2m\leqs2C(H_{s,t})\ll \beta n\le \frac{n}{2}-\sqrt{\gamma}n\leqs|V_i|$, both $G[V_0]$ and $G[V_1]$ have isolated vertices and {   $B=\emptyset$ by the definition of $B$ (hence in the following algorithm we always choose $B_0=B_i=\emptyset$ in each iteration)}. Since $\delta(G)\geqs\lf\frac{n}{2}\rf$, we have $|V_i|\ge \lf\frac{n}{2}\rf$ for $i=0,1$. %By considering these isolated vertices and note that $\delta(G)\geqs\lf\frac{n}{2}\rf$, we can easily conclude that $(V_0,V_1)$ is a balanced partition of $V(G)$,
Since $\{V_0, V_1\}$ is a partition of $V(G)$, we have $\lf\frac{n}{2}\rf\leqs|V_i|\leqs\lc\frac{n}{2}\rc$ for $i=0,1$. That is $\{V_0,V_1\}$ is a balanced partition of $V(G)$.  We will use the following algorithm to find enough many copies of $H_{s,t}$ in $G$. Initially, set $G^0=G$, $V_i^0=V_i$ and set $U_{i}^0$ to be the set of isolated vertices in $G[V_{i}^0]$ for $i=0,1$. Note that  $|V_i^0\setminus U_i^0|\le 2m\le 2C(H_{s,t})< \sqrt{\gamma}n$ for sufficiently large $n$.

\medskip

\noindent {\bf Algorithm 2}: Begin with $G^0,U^0_i~(i=0,1)$. Suppose that we have get $G^j$, $V_i^j$ and $U_i^j$ $(i=0,1)$ for some $j\geqs0$. Define $i^*=i^*(j)=\left\{\begin{aligned}
              0 & \mbox{\ \   if $|V_0^j|\geqs|V_1^j|$},\\
              1  & \mbox{\ \   otherwise}.
            \end{aligned}\right.
$

% if $|V_0^j|\geqs|V_1^j|$, otherwise we define $i^*=0$.\\
\noindent{\bf Step 1.} If there exists some vertex $x^j\in V_{i}^j$ with $\deg_{G[V_{i}^j]}(x^j)\geqs k$, then, apply Lemma \ref{LEMMA: Main technique} to $G^j$, we find a copy of $H_{s,t}$, say $H$,  with $V(H)\cap U_{i^*}^j\not=\emptyset$. Choose some vertex $u^j\in V(H)\cap U_{i^*}^j$. Let $G^{j+1}$ be the graph obtained from $G^j$ by deleting $u^j$ and $E(H)$. Set $V_{i^*}^{j+1}=V_{i^*}^j\setminus\{u^j\}$, $V_{1-i^*}^{j+1}=V_{1-i^*}^j$, $U_{i^*}^{j+1}=U_{i^*}^j\setminus\{u^j\}$, and $U_{1-i^*}^{j+1}=U_{1-i^*}^j$ for $i=0,1$. We stop Step 1 until there is no vertex of in-degree at least $k$ in $G^a$ for some integer $a$ and then turn to Step 2.

\medskip

\noindent{\bf Step 2.}  If there exists some vertex $x^j\in V_{i}^j$ satisfying inequality~(\ref{e1}) of Lemma \ref{LEMMA: Main technique}, then, apply Lemma \ref{LEMMA: Main technique} to $G^j$,  we find a copy of $H_{s,t}$, say $H$, with center $x^j$ and $V(H)\cap U_{1-i}^j\not=\emptyset$. { Set $u^j=x^j$ if $i=i^*$, otherwise choose any $u^j\in V(H)\cap U_{i^*}^j$.}
 Let $G^{j+1}$ be the graph obtained from $G^j$ by deleting $u^j$ and $E(H)$.
Set $V_{i^*}^{j+1}=V_{i^*}^j\setminus\{u^j\}$, $V_{1-i^*}^{j+1}=V_{1-i^*}^j$, $U_{i^*}^{j+1}=U_{i^*}^j\setminus\{u^j\}$ and $U_{1-i^*}^{j+1}=U_{1-i^*}^j$. We stop Step 2 if there is no vertex satisfying inequality~(\ref{e1}) in $G^b$ for some integer $b\ge a$.

\medskip
\noindent{{\textbf{Remark.}} Clearly, $b\leqs m/k<C(H_{s,t})/k$ can be bounded by a constant since in each iteration we find a copy of $H_{s,t}$ intersecting $E_G(V_0)\cup E_G(V_1)$ exactly $k$ edges and $C(H_{s,t})$ is a constant. So we always assume that $|V(G^j)|=n-j$ is sufficiently large and   $\ex(n-j,H_{s,t})=e(T_{2,n-j})+(k-1)^2$ for every $j=0,1,\cdots,b$ (by Theorem \ref{THM:MAIN THEOREM 1}). Also note that  in each iteration we delete one vertex from the bigger part, so $\{V^j_0,V^j_1\}$ is balanced for all $j\in \{0,1,\cdots,b\}$ since $\{V^0_0,V^0_1\}$ is balanced.}

\medskip
%\begin{claim}\label{CLM: Algorithm 2 can be successfully iterated}
\noindent{\bf Claim 3}	For any $j\in\{0,1,\cdots, b\}$, $e_{G^j}(u,V_{1-i}^j)>\lf\frac{n}{2}\rf-4C(H_{s,t})$ for each $u\in V_i^j$, $i=0,1$. In particular, Algorithm 2 can be successfully iterated.
%\end{claim}

\medskip
\noindent Let $u\in V_i^j$. Since in each previous iteration $u$ lost at most $\Delta(H_{s,t})+1=2k+1$ neighbors in $V_{1-i}$, we have
	$$e_{G^j}(u,V_{1-i}^j)\geqs e_{G^0}(u,V_{1-i})-(2k+1)b\geqs\lf\frac{n}{2}\rf-m-(2k+1)b\geqs\lf\frac{n}{2}\rf-4C(H_{s,t}).$$
	Another clear fact is that $|V_i^j\setminus U_i^j|\leqs2m\le 2C(H_{s,t})<\sqrt{\gamma}(n-j)$ for $i\in\{0,1\}$ and $j\in\{0,1,2,\cdots, b\}$. So Algorithm 2 can be successfully iterated.	
%\end{proof}

\medskip

\noindent{\bf Claim 4} $$p_{H_{s,t}}(G)\geqs\frac{e(G)-\ex(n,H_{s,t})}{e(H_{s,t})-1},$$ with equality holds if and only if $G\in\mathcal F_{n,s,t}$.

\medskip
\noindent
 {We prove  $p_{H_{s,t}}(G^j)\geqs\frac{e(G^j)-\ex(n,H_{s,t})}{e(H_{s,t})-1}$ inductively for $j=b,b-1,\dots,0$.
 Now since there is no vertex in $G^b$ satisfying inequality (1) of Lemma \ref{LEMMA: Main technique}, we must have $e(G^b)\leqs\ex(n-b,H_{s,t})$, otherwise, $e(G^b)>\ex(n-b,H_{s,t})$, then by Lemma \ref{LEMMA: Force to extremal graph}, we have $e(G^b)=\ex(n-b,H_{s,t})$, a contradiction. So we have $$p_{H_{s,t}}(G^b)\geqs0\geqs\frac{e(G^b)-\ex(n-b,H_{s,t})}{e(H_{s,t})-1},$$with equality holds if and only if $G^b\in\mathcal F_{n-b,s,t}$.
 	
 Now suppose that $$p_{H_{s,t}}(G^{j+1})\geqs\frac{e(G^{j+1})-\ex(n-j-1,H_{s,t})}{e(H_{s,t})-1}$$ holds for some $j+1\in[1,b]$. We show that the above inequality holds for $j$. By Algorithm 2,
 suppose that $G^{j+1}$ is obtained from $G^j$ by deleting $u^j$ and $E(H)$, where $H$ is a  copy of $H_{s,t}$. Clearly, we have
 $$e(G^{j+1})=e(G^j)-e(H)-\deg_{G^j}(u^j)+\deg_{H}(u^j).$$
 % and  $$\ex(n-j-1,H_{s,t})=\ex(n-j,H_{s,t})-\lf\frac{n-j}{2}\rf.$$
 	
 	If $u^j\in V(H)\cap U_{i^*}^j$, then $\deg_{G^j}(u^j)\leqs\lf\frac{n-j}{2}\rf$ by definition of $U$ and $i^*$. Note that  $\deg_H(u^j)\geqs\delta(H_{s,t})=2,$ and $\{V_0^j,V_1^j\}$ is balanced. Hence we have
 	\begin{eqnarray*}
 		p_{H_{s,t}}(G^{j})&\geqs&p_{H_{s,t}}(G^{j+1})+1\\
 		&\geqs&\frac{e(G^{j+1})-\ex(n-j-1,H_{s,t})}{e(H_{s,t})-1}+1\\
 		&\geqs&\frac{\big(e(G^j)-e(H)-\lf\frac{n-j}{2}\rf+2\big)-\big(\ex(n-j,H_{s,t})-\lf\frac{n-j}{2}\rf\big)}{e(H_{s,t})-1}+1\\
 		&=&\frac{e(G^{j})-\ex(n-j,H_{s,t})+1}{e(H_{s,t})-1}\\
 		&>&\frac{e(G^{j})-\ex(n-j,H_{s,t})}{e(H_{s,t})-1}.
 	\end{eqnarray*}
 	
 	If $u^j\notin V(H)\cap U_{i^*}^j$, then the case only happens in Step 2 when $u^j\in V_{i^*}^j$ is the center of $H$. So $\deg_H(u^j)=2k$. Meanwile, since it happens in Step 2, $u^j$ has in-degree at most $k-1$. So we have that $\deg_{G^j}(u^j)\leqs k-1+\lf\frac{n-j}{2}\rf$. Thus
 	\begin{eqnarray*}
 		p_{H_{s,t}}(G^{j})&\geqs&p_{H_{s,t}}(G^{j+1})+1\\
 		&\geqs&\frac{e(G^{j+1})-\ex(n-j-1,H_{s,t})}{e(H_{s,t})-1}+1\\
 		&\geqs&\frac{\left( e(G^j)-e(H)-(k-1+\lf\frac{n-j}{2}\rf)+2k\right )-\left(\ex(n-j,H_{s,t})+\lf\frac{n-j}{2}\rf\right)}{e(H_{s,t})-1}+1\\
 		&=&\frac{e(G^{j})-\ex(n-j,H_{s,t})+k}{e(H_{s,t})-1}\\
 		&>&\frac{e(G^{j})-\ex(n-j,H_{s,t})}{e(H_{s,t})-1}.
 	\end{eqnarray*}

Therefore, we can inductively conclude that $p_{H_{s,t}}(G^j)\geqs\frac{e(G^j)-\ex(n-j,H_{s,t})}{e(H_{s,t})-1}$ for all $j\in[0,b]$, with equality holds if and only if $b=0$ and $G=G^0\in\mathcal F_{n,s,t}$.  	The proof is now completed.
 	\qed
 }


\begin{thebibliography}{99}

  \bibitem{ImprovedError}
	P. Allen, J, B\"{o}ttcher, and Y. Person, An improved error term for minimum $H$-decompotions of graphs, {\it J. Combin. Theory Ser. B}, {\bf108} (2014), 92-101.
	
	\bibitem{Bollobas}
	B. Bollob\'{a}s, On complete subgraphs of different orders, {\it Math. Proc. Cambridge. Philos. Soc.} {\bf 79(1)} (1976) 19-24.
	
	


	\bibitem{Hanson}
	V. Chav\'atal, D. Hanson, Degrees and matchings, {\it J. Combin. Theory Ser. B}, {\bf 20(2)} (1976) 128--138.
	
	\bibitem {Wei}
	G. Chen, R. J. Gould, F. Pfender, and B. Wei, Extremal graphs for intersecting cliques, {\it J. Combin. Theory Ser. B} {\bf 89} (2003) 159--171.
	
	\bibitem{Erdos-66}
	P. Erd\H{o}s, On some new inequalities concerning extremal properties of graphs, in: Theory of Graphs, Proc. Colloq., Tihany, 1966, Academic Press, New York, 1968, pp. 77-81.
	
	\bibitem{Erdos-95}
	P. Erd\H{o}s, Z. F\"uredi, R. J. Gould, D. S. Gunderson, Extremal graphs for intersecting triangles, {\it J. Combin. Theory Ser. B} {\bf 64(1)} (1995) 89--100.

\bibitem{intersections}
	P. Erd\H{o}s, A. W. Goodman, L. P\'{o}sa, The representation of a graph by set intersections, {\it Canad. J. Math.} {\bf18}(1966) 106--112.
	
	\bibitem{Glebov-2011}
	R. Glebov. Extremal graphs for clique-paths. arXiv:1111.7029v1, 2011.

\bibitem{Dec for Fkr}
	X. Hou, Y. Qiu and B. Liu, Decompositions of graphs into $(k,r)$-fans and single edges, arXiv: 1510.00811
	
	\bibitem{Ex for Ckq}
	X. Hou, Y. Qiu and B. Liu, Extremal graph for intersecting odd cycles, {\it Electron. J. Combin.}, 23(2) (2016), P2.29.
	
	\bibitem{Liu-2013}
	H. Liu, Extremal graphs for blow-ups of cycles and trees, {\it The Electron. J. of combin.} 20(1) (2013), P65.
	
  \bibitem{k-fan}
	H. Liu, T. Sousa, Decompositions of Graphs into Fans and Single Edges, manuscript. {\it J. Graph Theory}, doi:10.1002/jgt.22069.
	
	\bibitem{edge-critical case}
	L. \"{O}zkahya and Y. Person, Minimum $H$-decomposition of graphs: edge-critical case, {\it J. Combin. Theory Ser. B}, {\bf102(3)} 715--725, 2012.
	
	\bibitem{general}
	O. Pikhurko and T. Sousa, Minimum $H$-decompositions of graphs,  {\it J. Combin. Theory Ser.B}, {\bf 97(6)} (2007) 1041--1055.
	


   \bibitem{Simonovits-66}
	M. Simonovits, A method for solving extremal problems in graph theory, stability problems, in: Theory of Graphs, Proc.
	Colloq., Tihany, 1966, Academic Press, New York, 1968, pp. 279-319.
	
	
	
	\bibitem{5-cycles}
	T. Sousa, Decompositions of graphs into 5-cycles and other small graphs, {\it Electron. J. Combin.,} 12 (2005), p49.
	
	\bibitem{7-cycle}
	T. Sousa, Decompositions of graphs into cycles of length seven and single edges, {\it Ars Combin.,}  119 (2015), 321-329.
	
	\bibitem{clique-extension}
	T. Sousa, Decompositions of graphs into a given clique-extension, {\it Ars Combin.,} {\bf100} (2011), 465--472.
	
	
\end{thebibliography}
\end{document}